\documentclass[12pt]{amsart}
 \usepackage{graphicx}

\usepackage{amsmath,graphics}
\usepackage{amsfonts,amssymb,xypic}
\usepackage{enumitem}
\usepackage{color}

\newcounter{commentcounter}

\theoremstyle{plain}
\newtheorem*{theorem*}{Theorem}
\newtheorem*{lemma*} {Lemma}
\newtheorem*{corollary*} {Corollary}
\newtheorem*{proposition*} {Proposition}
\newtheorem{theorem}{Theorem}[section]

\newtheorem{corollary}[theorem]{Corollary}
\newtheorem{proposition}[theorem]{Proposition}
\newtheorem{conjecture}[theorem]{Conjecture}
\newtheorem{question}[theorem]{Question}

\theoremstyle{remark}
\newtheorem*{remark}{Remark}

\theoremstyle{definition}

\def\tpm{[t^{\pm 1}]}
\def\tautwo{\tau^{(2)}}

\def \R {\Bbb{R}}
\def \Z {\Bbb{Z}}
\def \C {\Bbb{C}}
\def \F {\Bbb{F}}

\def\op{\operatorname}

\def\det{\op{det}}

\def \NN{\mathcal{N}} \def \PP{\mathcal{P}}

\def\wti{\widetilde}
\def\sl{\op{SL}}
\def\op{\operatorname}

\def\gl{\op{GL}}
\def\Q{\Bbb{Q}}
\def\id{\op{id}}

\def\Z{\Bbb{Z}}
\def\C{\Bbb{C}}

\def\part{\partial}
\def\ll{\langle}
\def\rr{\rangle}

\def\a{\alpha}
\def\g{\gamma}
\def\tor{\mbox{Tor}}
\def\bp{\begin{pmatrix}}

\def\sm{\setminus}
\def\aut{\mbox{Aut}}
\def\ep{\end{pmatrix}}
\def\bn{\begin{enumerate}}

\def\en{\end{enumerate}}
\def\ba{\begin{array}}
\def\ea{\end{array}}

\def\S{\Sigma}

\def\a{\alpha}

\def\fr12{\frac{1}{2}}

\def\be{\begin{equation} }
 \def\ee{\end{equation}}

\def\co{\colon}

\def\G{\Gamma}

\def\K{\Bbb{K}}

\def\vol{\mbox{Vol}}
\def\genus{\op{genus}}

\def\zt{\Z[t^{\pm 1}]}

\def\modd{\, mod \, }

\def\cmtbf#1{} \def\cmt#1{}

\begin{document}

\title{Three flavors of twisted invariants of knots}

\author{J\'er\^ome Dubois}

\address{Institut de Math\'ematiques de Jussieu - Paris Rive Gauche\\
 Universit\'e Paris Diderot-Paris 7\\
  UFR de Math\'ematiques, B\^atiment Sophie Germain Case 7012\\
  75205 Paris Cedex 13\\
   France}
\email{dubois@math.jussieu.fr}

\author{Stefan Friedl}
\address{Fakult\"at f\"ur Mathematik\\ Universit\"at Regensburg\\93040 Regensburg\\   Germany}
\email{sfriedl@gmail.com}

\author{Wolfgang L\"uck}
\address{Mathematisches Institut\\ Universit\"at Bonn\\
Endenicher Allee 60\\ 53115 Bonn\\ Germany}
\email{wolfgang.lueck@him.uni-bonn.de}
   
\date{\today}

\begin{abstract}
The Alexander polynomial of a knot has been generalized in three different ways to give twisted invariants. The resulting invariants are usually referred to as twisted Alexander polynomials, higher-order Alexander polynomials and $L^2$-Alexander invariants of knots. We quickly recall the definitions and we summarize and compare some of their properties. We also report on work by the authors on $L^2$-Alexander torsions and we conclude the paper with several conjectures on $L^2$-Alexander torsions. 
 \end{abstract}

\maketitle

\section{Introduction}

Alexander \cite{Al28} introduced in 1928 the eponymous  polynomial $\Delta_K(t)\in \zt$ of a knot $K$ in the three--dimensional sphere $S^{3}$. 
In contrast to its mysterious twin, the Jones polynomial, the formal properties and the topological content of the Alexander polynomial and its many generalizations are for the most part well-understood.
For example,  Seifert \cite{Se34} showed that the Alexander polynomial can be normalized such that $\Delta_K(t^{-1})=\Delta_K(t)$ and $\Delta_K(1)=1$,
and that any polynomial satisfying these two conditions can be realized as the Alexander polynomial of a  knot. Furthermore, if $K_{1}$ and $K_{2}$ are oriented knots, then we can consider the connected sum $K_1\# K_2$ and we obtain the following equality
\be \label{equ:deltasum} \Delta_{K_{1}\# K_{2}}(t)=\Delta_{K_{1}}(t)\cdot \Delta_{K_{2}}(t).\ee
In terms of topological information we have for any knot $K$ the inequality
\be \label{equ:degdeltak} \deg(\Delta_K(t))\leq 2\,\mbox{genus}(K),\ee
where $\genus(K)$ denotes the minimal genus of a Seifert surface for $K$.
Also, if $K$ is a fibered knot, then 
\be \label{equ:deltakfib}  \deg(\Delta_K(t))= 2\,\mbox{genus}(K)\mbox{ and 
$\Delta_K(t)$ is monic.}\ee
The Alexander polynomial also contains information on  symmetries of knots \cite{Mu71,Hat81} and on knot concordance \cite{FM66,Ka78,FQ90}.
The original definition of the Alexander polynomial has  been extended to much more general settings and the generalizations of the Alexander polynomial 
have been effectively used in the study of links, 3-manifolds \cite{Mc02}, algebraic varieties and singularities \cite{Di92}.

Nonetheless, it is well-known that in each case the Alexander polynomial has only partial information.
For example, the fact that there exist (infinitely many) non-trivial  knots with trivial Alexander polynomial shows that Alexander polynomials are not a complete invariant and it also shows that Equation~(\ref{equ:degdeltak}) is in general not an equality.
Over the last years a huge effort has been put into finding invariants which generalize the Alexander polynomial and which contain more information, especially more topological information.

Arguably the most important and successful generalization is  Heegaard Floer homology due to Ozsv\'ath--Szab\'o \cite{OS04a,OS04b} and its offspring knot Floer homology  \cite{Ras03,OS04c}
and sutured Floer homology \cite{Ju06}.  
These invariants always detect the knot genus, and more generally the Thurston norm, and they also detect fibered knots and 3-manifolds
\cite{Ni07,Gh08,Ju08,AN09,Ni09a,Ni09b,AlFJ13}. The Heegaard Floer setup has been amazingly effective in dealing with many problems in topology,
it is impossible for us to list all the results. We therefore refer instead to the recent beautiful survey paper by Juh\'asz \cite{Ju13}.
Despite, or arguably because of, the power of Heegaard Floer invariants there are at least two issues. For one, despite some progress \cite{MaOS09,SaW10,LCSVV13} the invariants tend to be difficult to calculate in more complicated situations. Furthermore, the Heegaard Floer invariants can not be generalized to invariants of higher-dimensional manifolds or to invariants of groups.

The quest for useful generalizations of the Alexander polynomial is therefore not over yet. 
In recent years there has been a lot of interest in twisted versions of the Alexander polynomial.
These twisted invariants come in three flavors:
\bn[label=(\alph*)]
\item The \emph{twisted Alexander polynomial} introduced by Lin \cite{Lin01} and Wada \cite{Wa94} associates to an oriented knot {$K$} and a {linear} representation {of its group} $\a\colon \pi_1(S^3\sm K)\to \sl(k,\F)$ over a commutative field $\F$ an invariant $\Delta_K^\a(t)\in \F\tpm$. 
\item The \emph{higher-order Alexander polynomial} of Cochran \cite{Co04} associates to an oriented knot $K$ and an epimorphism $\g\colon \pi_1(S^3\sm K)\to \G$ onto a  torsion-free elementary-amenable group $\G$ an invariant $\Delta_K^\g$ which is a Laurent polynomial with coefficients in a certain skew field.
\item The \emph{$L^2$-Alexander invariant} of Li--Zhang \cite{LiZ06a,LiZ06a} associates to an oriented knot {$K$ in $S^{3}$}  a function $\Delta_{K}^{(2)}\colon \R_{>0}\to \R_{\geq 0}$. 
\en
The goal of this paper is to discuss these three twisted invariants.
More precisely, for each invariant we will outline the definition 
and we will recall some of the key properties, with a special focus on the relationship to the knot genus and fiberedness.

The first two invariants are by now fairly well-known, but the third invariant is hitherto little studied. We will  introduce a slight variation on the $L^2$-Alexander invariant, namely  the  $L^2$-Alexander torsion
 $\tautwo(K)(t)$ of a knot which is by definition a function $\R_{>0}\to \R_{\geq 0}$.
It follows from the definitions that $\tautwo(K)(1)$ equals the usual $L^2$-torsion of the knot exterior which by work of L\"uck--Schick \cite{LS99}
 implies that $\tautwo(K)(1)$ is in fact a repackaging of arguably the most important geometric invariant of a knot complement, namely the volume of a knot. More precisely, given a knot $K$ they showed that 
\[\tautwo(K)(1)=\exp\left(\frac{1}{6\pi}\vol(K)\right),\]
 where we define the  volume of $K$ as 
\[ \vol(K):=\sum \mbox{volumes of the hyperbolic pieces in the JSJ decomposition of {$S^3\sm K$}}.\]
Put differently,
the function $t\mapsto \tautwo(K)(t)$ can be viewed as a deformation of the volume of a knot. Our main goal will then be to make the following conjecture precise and to give some evidence towards it.

\begin{conjecture}\label{conj:l2}
Let $K$ be a knot.
\bn
\item[$(1)$] The full $L^2$-Alexander torsion $\tautwo(K)(t)$ determines the genus of $K$.
\item[$(2)$] The full $L^2$-Alexander torsion determines whether or not $K$ is fibered.
\item[$(3)$]  If $K$ is fibered, then  $\tautwo(K)(t)$  determines the entropy of the monodromy.
\en
\end{conjecture}
\medskip

The paper is organized as follows.
In Section \ref{section:untwisted} we  recall the recasting of the Alexander polynomial as a torsion invariant, which we refer to as the Alexander torsion. This point of view was introduced by Milnor \cite{Mi62,Mi66} and exploited very successfully by Turaev \cite{Tu86,Tu01,Tu02a}.
In our discussion of twisted invariants we will in fact discuss twisted versions of the Alexander torsion. The differences to the aforementioned Alexander polynomials and invariants are minimal, but it is a well-established fact that Reidemeister
 torsion has better formal properties than  orders of modules. 
In Section \ref{section:twialex} we recall the definition and main properties of the twisted Alexander torsion and  in Section \ref{section:higherorder} 
we do the same for higher-order Alexander torsion of knots.
Finally in Section \ref{section:l2torsion} we will turn our attention to the $L^2$-Alexander torsion of a knot. We first give an outline of the  key properties of the Fuglede-Kadison determinant and  of $L^2$-torsions.
  We then introduce the $L^2$-Alexander torsions of a knot and we state some of the results and computations obtained by the first author and Wegner \cite{DubW10,DubW13}, Ben-Aribi \cite{BA13a,BA13b} and by the authors \cite{DubFL14a,DubFL14b}. We conclude this paper with a long list of open questions on the $L^2$-Alexander torsions of knots. In particular we will discuss Conjecture \ref{conj:l2} in more detail.

\subsection*{Acknowledgments.}
The second author wishes to thank the organizers of the Fourth Conference of the Tsinghua Sanya International Mathematics Forum and in particular the organizers of the workshop `speculations and wild conjectures in low dimensional differential topology'. The second author also gratefully acknowledges the support provided by the SFB 1085 `Higher
Invariants' at the University of Regensburg, funded by the Deutsche
  Forschungsgemeinschaft (DFG).
The first two authors would like to thank IMJ-PRG for its hospitality during the stay of the second author in Paris in March 2012. 
The paper is financially supported by the Leibniz-Award of the third author granted by the  DFG.
We are also grateful to Fathi Ben Aribi  and  Stefano Vidussi for helpful comments.

\section{The Alexander polynomial and Alexander torsion}\label{section:untwisted}

The Alexander polynomial has many different definitions, the equivalence of which is not always entirely obvious. We now recall the definition which for  most theoretical purposes is the most useful and which also lends itself most easily to generalizations to other classes of manifold and groups.

Let $K\subset S^3$ be an oriented knot. 
We denote by  $\nu K$  an open tubular neighborhood of $K$
and we denote by $X_K:=S^3\sm \nu K$ {the knot exterior}. We refer to $\pi_K=\pi_1(S^3\sm \nu K)$ as the group of $K$ and we denote by $\wti{X_K}$ the universal cover of $X_K$. Throughout this paper we always think of $X_K$ as equipped with a CW-structure. Note that $\pi_K$ acts via deck transformations
on the left on $\wti{X_K}$ and thus induces a left $\Z[\pi_K]$-module structure on  $C_*(\wti{X_K}): = C_*(\wti{X_K}; \Z)$.
Using the canonical involution on $\Z[\pi_K]$ we henceforth view
$C_*(\wti{X_K})$ as a right $\Z[\pi_K]$-module.

We denote by $\phi_K\colon\pi_K\to \ll t\rr=\Z$ the abelianization map which sends the oriented meridian of $K$ to $t$. We then consider the chain complex
\[ C_*(\wti{X_K})\otimes_{\Z[\pi_K]}\zt\]
of $\zt$-modules.
Here $g\in \pi_K$ acts on $\zt$ by multiplication by $t^{\phi_K(g)}$.
We then write 
\[ H_k(X_K;\zt)=H_k\left(C_*(\wti{X_K})\otimes_{\Z[\pi_K]}\zt\right).\]
The \emph{Alexander polynomial of $K$} is then defined as the order of the $\zt$-module $H_1(X_K;\zt)$, i.e., it is  defined as
\[ \Delta_K(t):=\op{order}_{\zt}\left(H_1(X_K;\zt)\right).\]
We refer to \cite{Tu01,Hi12} for details on orders. It follows in particular from the theory of orders that $\Delta_K(t)$ is well-defined up to multiplication by a unit in $\zt$, i.e., up to multiplication by an element of the form $\pm t^k$ with $k\in \Z$.

We now turn to the Alexander torsion of a knot. We consider the chain complex
\[ C_*(\wti{X_K})\otimes_{\Z[\pi_K]}\Q(t)\]
of $\Q(t)$-modules.
By  picking  an order of the cells, an orientation of each cell, and by picking a lift of each cell of $X_K$ to $\wti{X_K}$
we can view the above chain complex as a chain complex of based free $\Q(t)$-modules.
{One can further observe that this chain complex is acyclic (see~\cite{Mi62}), so that it is possible to compute its Reidemeister torsion.}
The Alexander torsion of $K$, which is sometimes also referred to as Milnor torsion, is then defined as 
\[ \tau(K)(t):=\tor\left(C_*(\wti{X_K})\otimes_{\Z[\pi_K]}\Q(t)\right)\in \Q(t).\]
We refer to \cite{Mi66,Tu86,Tu01} for more on torsions of chain complexes.
This invariant depends on the choice of the CW-structure, the ordering of the cells, the orientation of the cells and the choice of lifts of the cells to the universal cover.
Nonetheless, it follows from standard arguments that $\tau(K)(t)$ is well-defined up to multiplication by an element of the form $\pm t^k, k\in \Z$. 

The Alexander torsion $\tau(K)(t)$ thus has the same indeterminacy as the Alexander polynomial. In fact Milnor \cite{Mi62,Mi66}  showed that the Alexander  torsion $\tau(K)(t)\in \Q(t)$  satisfies the following equality 
\[ \tau(K)(t)=(1-t)^{-1}\cdot \Delta_K(t).\]
Even though the difference between $\tau(K)(t)$ and $\Delta_K(t)$ is marginal, Turaev \cite{Tu86,Tu01,Tu02a} showed that this  shift in point of view 
greatly simplifies many proofs and that it is `the right point of view'. It is also much easier to generalize $\tau(K)(t)$ to twisted settings and to prove properties of these new invariants.

In the remainder of this section we translate several of the aforementioned properties of the Alexander polynomial into properties of the Alexander torsion. For example, the fact that $\Delta_K(t)$ can be chosen to be symmetric is equivalent to the statement that for any representative 
of $\tau(K)(t)$ we have
\be \label{equ:tausym} \tau(K)(t^{-1})=-t^l\tau(K)(t)\ee
for some odd $l$.
In the following we define the degree of a non-zero polynomial
$p(t)=\sum_{i=k}^l a_it^i$ with $a_k\ne 0$ and $a_l\ne 0$ as $\deg(p(t))=l-k$.
For a non-zero rational function $f(t)=p(t)/q(t)$ we define its degree as
$\deg(f(t))=\deg(p(t))-\deg(q(t))$. We extend this to $\deg(0):=-\infty$. 
Note that with this convention Equality~(\ref{equ:tausym}) implies that 
$\deg(\tau(K)(t))$ is odd. Furthermore,  Inequality~(\ref{equ:degdeltak}) translates into 
\be \label{equ:degtau} \deg(\tau(K)(t))\leq 2\,\mbox{genus}(K)-1.\ee
We furthermore say that a rational function is \emph{monic} if it is the quotient of two monic polynomials, i.e., polynomials for which the top coefficient is $\pm 1$. If $K$ is a fibered knot, then the conditions stated in {Equation}~(\ref{equ:deltakfib}) now translate into
\be \label{equ:taufib}  \deg(\tau(K)(t))= 2\,\mbox{genus}(K)-1\mbox{ and 
$\tau(K)(t)$ is monic.}\ee

\section{Twisted Alexander torsion}\label{section:twialex}

\subsection{Definition}
Let $K\subset S^3$ be an oriented knot and 
let $\a\colon \pi_K\to \sl(k,\F)$ be a representation over a commutative field $\F$.   We consider the chain complex
\[ C_*(\wti{X_K})\otimes_{\Z[\pi_K]}(\F(t)\otimes \F^k)\]
of $\F(t)$-modules, 
where $g\in \pi_K$ acts again on $\F(t)$ by multiplication by $t^{\phi_K(g)}$ and it acts furthermore on $\F^k$ via the representation $\a$. 
We pick a basis for $\F^k$ and lifts of the cells of $X_K$ to $\wti{X_K}$. The  tensor products of the basis elements and the lifts  then turn
the above into a based  $\F(t)$-chain complex. If the above twisted chain complex
is non-acyclic, then we write $\tau(K,\a)(t)=0$.
Otherwise  the \emph{twisted Alexander torsion  of $(K,\a)$} is  defined as 
\[ \tau(K,\a)(t):= \tor\left(C_*(\wti{X_K})\otimes_{\Z[\pi_K]}(\F(t)\otimes \F^k)\right)\in \F(t).\]
Note that  $\tau(K,\a)(t)\in \F(t)$ is
  well-defined up to multiplication by an element of the form $\pm t^l$, $l\in \Z$.
In analogy to (\ref{equ:tausym}) it was shown in \cite{Ki96,HiSW10,FrKK12,Hi12}  that if $\a$ is a unitary representation or if $\a$ is a representation taking values in $\sl(2,\C)$, then for any representative of $\tau(K,\a)(t)$ we have
\[ \tau(K,\a)(t^{-1})=(-t)^l\tau(K,\a)(t)\]
where $l\equiv k\modd 2$. In particular, if $\tau(K,\a)(t)\ne 0$, then the parity of $\deg(\tau(K,\a)(t))$  is the same as the parity of  $k$.  We refer to \cite{Mo11,MS13} for more on degrees of twisted Alexander torsions.

Kitano \cite{Ki96}  showed that the above definition of the twisted Reidemeister torsion  $\tau(K,\a)(t)\in \F(t)$ of an oriented knot $K$
is  equivalent to Wada's invariant \cite{Wa94}, which in turn is closely related to the twisted Alexander torsion introduced by  Lin \cite{Lin01}. 
We refer to \cite{Ki96,KiL99a,FrV10} for the precise relationship between the twisted Reidemeister torsions and various related invariants.

\subsection{Applications and properties}

The twisted Reidemeister torsion  $\tau(K,\a)(t)$ and its generalizations to links, 3-manifolds and groups have been studied extensively over the last years.
These invariants have in particular been applied to  knot concordance \cite{KiL99a,KiL99b,Ta02,HerKL10,ColKL13}, periodicity of  knots \cite{HiLN06,Ell08},
detecting various types of  knots and links \cite{SiW06,FrV07,FrV13},
a certain partial ordering on knots \cite{HoKMS11,HoKMS12} 
and the study of singularities in particular and algebraic geometry in general \cite{CogF07,Coh08}.
We refer to the survey paper \cite{FrV10} for more details. 

In the following we will highlight a few results which have appeared after the survey paper \cite{FrV10} was written and we also highlight a few results which will be of interest to us when we compare the three flavors of twisted invariants.

Alexander polynomials and its generalizations are particularly suitable for the study of the knot genus and fibered  knots. The following theorem says in particular that twisted Alexander torsions detect the genus of a given  knot.

\begin{theorem}\label{thm:twialexdetectsgenus}
Let $K\subset S^3$ be an oriented knot. Then for any representation
$\a\colon \pi_K\to \sl(k,\F)$ over a commutative field  we have
\[ \deg( \tau(K,\a)(t))\leq k(2\genus(K)-1).\]
Furthermore,  there exists a representation
$\a\colon \pi_K\to \sl(k,\F)$ over a commutative field such that 
\[ \deg( \tau(K,\a)(t))=k(2\genus(K)-1).\]
\end{theorem}

Here, the first statement is proved in   \cite{Tu02a,FrK06,Fr14} whereas the second statement is proved in \cite{FrV12b}. The latter result builds
 on the Virtually Compact Special Theorem of  Wise \cite{Wi09,Wi12a,Wi12b} and  Przytycki--Wise \cite{PW12}  and on the Virtual Fibering Theorem
 of Agol \cite{Ag08} (see also \cite{FrKt14}).

Similarly we have the following theorem which says that twisted Alexander torsions detect whether or not a given  knot is fibered.

\begin{theorem}\label{thm:twialexdetectsfib}
Let $K\subset S^3$ be an oriented knot. If $K$ is fibered, then for any representation
$\a\colon \pi_K\to \sl(k,\F)$ over a commutative field we have
\[ \deg( \tau(K,\a)(t))=k(2\genus(K)-1)\]
and $\tau(K,\a)(t)$ is monic. Conversely, if $K$ is not fibered, then 
 there exists a representation
$\a\colon \pi_K\to \sl(k,\F)$ over a commutative field such that 
\[  \tau(K,\a)(t)=0.\]
\end{theorem}

Here, the first part of the theorem  was  shown in \cite{GKiM05} (see also \cite{GM03,Ch03,Fr14})
and the second part  was shown in \cite{FrV12a} (see also \cite{FrV08,FrV11}), the proof of which  again builds  on the recent work of Wise \cite{Wi09,Wi12a,Wi12b}. 

The alert reader will have noticed that neither Theorem \ref{thm:twialexdetectsgenus}
nor Theorem \ref{thm:twialexdetectsfib} specifies the representation which detects the genus and which detects non-fiberedness.
This is not a problem for some of the theoretical applications of the theorems, e.g. in the applications to 
splittings of knot groups \cite{FrSW13}, symplectic 4-manifolds \cite{FrV08,FrV11,FrV12a}, rank gradients of 3-manifold groups \cite{DFV14}
and fundamental groups of non-fibered knots \cite{SiW09a,SiW09b}.

It is also straightforward to see that both theorems give rise to an algorithm which determines the genus and the fiberedness of a given  knot.
We refer to \cite{FrV12b} for details.

\subsection{Questions and conjectures}

As we mentioned in the previous section,  neither Theorem \ref{thm:twialexdetectsgenus}
nor Theorem \ref{thm:twialexdetectsfib} specifies the representation which has the desired property. If we want  efficient algorithms for determining the knot genus and fiberedness  it would be helpful to have more precise information regarding the representations.
The following conjecture was formulated in \cite{DunFJ12}.

\begin{conjecture}\label{conj:dfj12}
Let $K\subset S^3$ be a hyperbolic oriented knot. 
Let $\a\colon \pi_K\to \sl(2,\C)$ be a lift of the discrete and faithful representation, then
\[ \deg (\tau(K,\a)(t))=2(2\genus(K)-1).\]
Furthermore, $K$ is fibered if and only if $\tau(K,\a)(t)$ is monic.
\end{conjecture}

 A proof of this conjecture would result in an extremely fast algorithm for determining the knot and fiberedness of a hyperbolic knot.
In \cite{DunFJ12} the conjecture was verified for all hyperbolic knots up to 15 crossings. 
Further positive evidence towards this conjecture was also given in \cite{Mo12,KiKM13,MoT13,Tr13}. We refer to \cite{DunFJ12} for various other conjectures and open questions regarding the twisted Alexander torsion of hyperbolic knots. 

An elementary satellite knot construction shows that for any Alexander polynomial of a knot there exist infinitely many distinct knots with the same Alexander polynomial. Put differently, Alexander polynomials are far from being a complete invariant for knots. One of the first uses of twisted Alexander torsions, see \cite{Lin01,DunFJ12}, was to show that they can be used to distinguish the Conway knot from the Kinoshita-Terasaka knot. The following question is asked in \cite[Section~6.6.1]{Hi12}.

\begin{question}
Is it possible to distinguish any two distinct prime knots using twisted Alexander torsions?
\end{question}

Note that it requires some thought to make this question precise, since the twisted Alexander torsion $\tau(K,\a)(t)$ is an invariant of a knot together with a representation $\a$. 
Some evidence for a positive answer is provided by  \cite{SiW06,FrV07,FrV13} 
where it is shown that twisted Alexander torsions detect the unknot, the trefoil and the Figure-8 knot.  
We also refer to \cite{Ei07} for related ideas.

\section{Noncommutative invariants}\label{section:higherorder}

\subsection{Definition}

Given a knot  $K\subset S^3$ we say that an epimorphism
 $\g\colon \pi_K\to \G$ 
is \emph{admissible} if the abelianization $\phi_K\colon \pi_K\to H_1(X_K;\Z)\cong \Z$
factors through $\g$. Note that by replacing  $\g\colon \pi_K\to \G$ with $\g\times \phi_K$ we can turn any homomorphism into an admissible homomorphism.
Admissibility is thus not a big restriction.
Throughout this section  let $\g\co \pi_K\to \G$
be  an admissible epimorphism to a torsion-free elementary-amenable  group.
Since $\G$ is torsion-free and elementary-amenable the ring
$\Z[\G]$ admits by \cite{DLMSY03,KrLM88} an  Ore localization which we denote by $\K(\G)$. We then consider the chain complex
\[ C_*(\wti{X_K})\otimes_{\Z[\pi_K]}\K(\G)\]
of right $\K(\G)$-modules, 
where $g\in \pi_K$ acts on $\K(\G)$ via left multiplication by $\g(g)$.
We again pick a  lift of each cell in $X_K$ to $\wti{X_K}$. The corresponding cells then turn the above into a based  chain complex of right $\K(\G)$-modules.
If the chain complex is not acyclic, then we define $\tau(K,\g):=0$.
Otherwise we consider the corresponding Whitehead-Reidemeister torsion
\[ \tau(K,\g):= \tau\left(C_*(\wti{X_K})\otimes_{\Z[\pi_K]}\K(\G)\right)\,\,\in\,\,  K_1(\K(\G)).\]
We refer to \cite{Ro94} for the definition of the $K_1$-group of a ring.
Note that if we write $\K(\G)^\times=\K(\G)\sm \{0\}$, then the Dieudonn\'e {determinant}
 induces  by \cite{Ro94} a canonical isomorphism 
\[ K_1(\K(\G))\xrightarrow{\cong}\K(\G)^\times_{ab}:=\K(\G)^\times/[\K(\G)^\times,\K(\G)^\times].\]
The \emph{higher-order Alexander torsion  of $(K,\a)$} is then defined as the image of $\tau(K,\g)\in \{0\}\cup \K(\G)^\times_{ab}$. It is well-defined up to multiplication by an element of the form $\pm g, g\in \G$. 
The higher-order Alexander torsion was first defined in \cite{Fr07}, it is a slight variation of the higher-order Alexander polynomial introduced by Cochran \cite{Co04} and Harvey \cite{Har05}. 

It is perhaps puzzling at first glance that no `$t$' appears in this definition at all. It is thus a priori not clear why this should be viewed as a higher-order \emph{Alexander} torsion. Also note that the higher-order Alexander torsion takes values in a rather unwieldy algebraic object.  
This has caused serious problems and in fact the only useful invariant which has ever been extracted is the degree of $\tau(K,\g)$. 
In the following we will thus recall the definition of $\deg(\tau(K,\g))$.

By the admissibility of $\g$ the abelianization epimorphism $\phi_K\colon \pi_K\to \Z$ factors through $\g$. We denote the resulting epimorphism $\G\to \Z$ by $\phi_K$ as well.
 Given a non-zero $p=\sum_{g\in \G} a_gg\in \Z[\G]$ we now write
\[ \deg(p):=\max\{ \phi_K(g)-\phi_K(h)\,|\,a_g\ne 0\mbox{ and }a_h\ne 0\}.\]
Furthermore, for $pq^{-1}\in \K(\G)$ with $p,q\in \Z[\G]$ and $p,q\ne 0$ we define
\[ \deg(pq^{-1})=\deg(p)-\deg(q).\] 
We again extend this to $\deg(0):=-\infty$. It follows from $\deg(fg)=\deg(f)+\deg(g)$ that 
\[ \deg(\tau(K,\g))\in \Z \cup \{-\infty\}\]
is well-defined. As an example, if we take  $\g=\phi_K$ to be  the abelianization, then $\tau(K,\phi_K)=\tau(K)(t)$ `on the nose', and $\deg(\tau(K,\phi_K))=\deg(\tau(K)(t))$. 

The fact that $\tau(K,\g)$ always has a degree is justification enough for us to refer to it as an Alexander invariant. In \cite{FrK08,FrKK12} Equality~(\ref{equ:tausym}) was generalized to any $\g$, more precisely, it was shown that $\deg(\tau(K,\g))$ is always odd.

\subsection{Applications and properties}

The higher-order invariants, and its generalizations to more general manifolds have been studied in the context of knot concordance \cite{CocT08}, singular plane curves \cite{LeM06,LeM08}, Morse-Novikov theory \cite{Kiy10} and homology cobordisms of surfaces \cite{Sa06,Sa08}.
Arguably the greatest interest in higher-order invariants stems from their connection to the knot genus and the Thurston norm.
It is perhaps not entirely surprising that higher-order Alexander torsions give lower bounds on the knot genus. Indeed,  the following theorem was proved in \cite{Co04}, with generalizations and extensions given in \cite{Tu02b,Har05,FrH07,Fr07}.

\begin{theorem}\label{thm:higherordergenus}
Let $K\subset S^3$ be an oriented knot. Then for any admissible
epimorphism $\g\colon \pi_K\to \G$ onto a torsion-free elementary-amenable group  we have
\[ \deg( \tau(K,\g))\leq 2\genus(K)-1.\]
\end{theorem}

In general higher-order invariants are very difficult to calculate,
see e.g. \cite{Ho13} for some algorithms and a discussion of the inherent difficulties. Nonetheless, the examples given in \cite{Co04,Har05,Ho13} show that higher-order invariants give very powerful lower bounds on the knot genus. 

From a purely theoretical point of view the most interesting fact about higher-order Alexander torsions is the following theorem, due to Cochran \cite{Co04}, with extensions given in \cite{Har06,Fr07}, which says loosely speaking that `the bigger the quotient, the better the lower bound on the knot genus'. More precisely, the following theorem holds.

\begin{theorem}\label{thm:higherorderincrease}
Let $K\subset S^3$ be an oriented knot. Let  $\g\colon \pi_K\to \G$
be an epimorphism onto a torsion-free solvable group
 group and 
let $\delta\colon \G\to \Omega$ be another epimorphism  onto a torsion-free
solvable group such that $\g\circ \delta$ is admissible.  Then the following inequality holds:
\[ \deg( \tau(K,\g))\geq \deg( \tau(K,\g\circ \delta)).\]
\end{theorem}

Recall  that if $\g\colon \pi_K\to \G$ is an admissible epimorphism, then by definition the abelianization
$\phi_K$ factors through $\g$. It thus follows from Theorem \ref{thm:higherorderincrease} that the degree of a higher-order Alexander torsion is always at least the degree of the ordinary Alexander torsion.

Finally, we turn to fibered knots. If  $K\subset S^3$ is a fibered  knot,
then for any admissible epimorphism  $\g\colon \pi_K\to \G$
onto a torsion-free elementary-amenable we have by \cite{Co04,Har05,Fr07}
the following equality:
\[ \deg( \tau(K,\g))=2\genus(K)-1.\]
This gives only very limited information on fiberedness. For example, if $K$ is a non-fibered knot with $\deg(\Delta_K(t))=2\genus(K)$, then it follows from Theorem
\ref{thm:higherorderincrease} that degrees of higher-order torsion can not detect that $K$ is not fibered. The problem with higher-order Alexander torsions 
is that there is no established notion of `monicness'. One way around this problem is to define `monicness' in a radically different way, namely as the vanishing of a non-commutative Novikov homology as introduced in \cite{Si87}.
This point of view will be discussed in more detail in \cite{Fr15}. 
We also refer to \cite{GS11} for another approach to using noncommutative invariants for detecting non-fibered knots.

\subsection{Questions and conjectures}

The conclusion of Theorem \ref{thm:higherorderincrease} seems to suggest that the degree of the higher-order Alexander torsion corresponding to an 
epimorphism   $\g\colon \pi_K\to \G$ onto a torsion-free elementary-amenable group is the optimal lower bound on the genus that one can obtain from twisted invariants where the twisting factors through $\g$. 
We therefore propose the following conjecture.

\begin{conjecture}
Let $K\subset S^3$ be an oriented knot. Let  $\g\colon \pi_K\to \G$
be an epimorphism onto a torsion-free elementary-amenable
 group and let $\delta\colon \G\to \gl(k,\F)$ be a representation over a commutative field. Then the following inequality holds:
\[ \deg( \tau(K,\g))\geq \frac{1}{k}\deg( \tau(K,\g\circ \delta)(t)).\]
\end{conjecture}

We next turn to the question on whether higher-order Alexander torsions can detect the genus of a knot. 
The invariants of Cochran \cite{Co04,Har05} were initially defined for epimorphisms onto PTFA groups, which are special classes of  torsion-free solvable groups. 
It is straightforward to see that if $K$  is a knot with trivial Alexander polynomial, then the abelianization is the only  epimorphism onto a non-trivial torsion-free solvable group. 
It follows that all the higher-order invariants of $K$  corresponding  to epimorphisms onto torsion-free solvable groups are equal to $(1-t)^{-1}$,
in particular they can not determine the genus of $K$.

There is hope though if we consider
higher-order invariants corresponding to epimorphisms onto torsion-free elementary-amenable groups. In fact we propose the following conjecture.

\begin{conjecture}
For any knot $K$ there exists an epimorphism $\g\colon \pi_K\to\G$
onto a  torsion-free elementary-amenable group, such that 
\[ \deg(\tau(K,\g))=2\genus(K)-1.\]
\end{conjecture}

In order for this conjecture to have a chance to be true we need that 
knot groups have `enough' epimorphisms onto  torsion-free elementary-amenable group. Before we state our next conjecture we recall that if $\PP$ is a property of groups, then a group $\pi$ is called \emph{residually $\PP$} if given any $g\in \pi$ there exists an epimorphism $\g\colon \pi\to \G$ such that $\g(g)$ is non-trivial and such that $\G$ has Property $\PP$. 

We propose the following conjecture, which also appears in \cite{AsFW13} as a question.

\begin{conjecture}\label{qu:restfea}
Given any knot $K$ the group $\pi_K$ is residually torsion-free elementary-amenable.
\end{conjecture}

Some evidence for a positive answer to this conjecture is provided
by the work of Przytycki-Wise \cite{PW12} and Wise \cite{Wi09,Wi12a,Wi12b}
which implies that any knot group is virtually residually torsion-free nilpotent, i.e., 
any knot group admits a finite index normal subgroup which is residually torsion-free nilpotent. (We refer to \cite{AsFW13} for precise references.)

\section{The $L^2$-Alexander torsion}\label{section:l2torsion}

In this last section we will discuss $L^2$-Alexander torsions.
The original definition of  $L^2$-Alexander {invariants} was given by Li-Zhang \cite{LiZ06a,LiZ06b,LiZ08}  and so far it has been studied only in a few papers \cite{Du11,DubW10,DubW13,BA13a,BA13b}. 
Since this invariant  is a relative newcomer and still largely unknown we will discuss this invariant in more detail than the previous two twisted invariants.

\subsection{Definition of the $L^2$-torsion of a chain complex}

Before we start with the definition of the $L^2$-Alexander torsion we need to recall some key properties of the Fuglede-Kadison determinant and the definition of the $L^2$-torsion of a chain complex of $\R[\G]$-modules. Throughout the section we refer to \cite{Lu02} and to \cite{DubFL14a} for details and proofs.

Throughout this section let $\G$ be a group.
Let $A$ be a   matrix over $\R[\G]$.
Then there exists the notion of  $A$ being of `determinant class'. 
(To be slightly more precise, we view the $k\times l$-matrix $A$ as a homomorphism $\NN(\G)^l\to \NN(\G)^k$, where $\NN(\G)$ is the von Neumann algebra of $\G$, and then there is the notion of being of `determinant class'.)
We treat this entirely as a black box,
but we note that if $\G$ is residually amenable, e.g. if $\G$ is a 3-manifold group \cite{Hem87} or if $\G$ is solvable, then 
by \cite{Lu94,Sc01,Cl99,EleS05} any matrix over $\Q[\G]$ is of determinant class. 
It is in fact possible that all matrices which appear in our setup are of determinant class.

Let $A$ be a matrix over $\R[\G]$. 
(Note that we do not assume that $A$ is a square matrix.)
If $A$ is not of determinant class then for the purpose of this paper we define $\det_\G(A)=0$. On the other hand, if $A$ is of determinant class, then we define
\[ \det_\G(A):=\mbox{Fuglede-Kadison determinant of $A$}.\]

We will not provide a formal definition of the Fuglede-Kadison determinant, the non specialist reader can think of it as a continuous generalization of the usual determinant of an operator with a finite spectrum to an operator with an infinite spectrum.
Furthermore, in this paper we will not provide any proofs, but 
to give the reader a flavor of the Fuglede-Kadison determinant
we now list some properties.
\bn[label=(\alph*)]
\item If $A$ is of determinant class, then $\det_\G(A)>0$.
\item If $\G=\{e\}$ is the trivial group and $A$ is a square matrix over $\R=\R[\{e\}]$ such that the ordinary determinant $\det(A)\in \R$ is non-zero, then
it follows from \cite[Example~3.12]{Lu02} that  $\det_{\{e\}}(A)=|\det(A)|$.
\item If $A$ is a matrix over $\R[\G]$ and if $\G$ is a subgroup of a group $G$,
then $\det_\G(A)=\det_G(A)$. 
\item The Fuglede-Kadison determinant stays unchanged under the following operations.
\bn
\item[(i)] Adding a column of zeros or a row of zeros to a matrix.
\item[(ii)] Swapping two columns or rows.
\item[(iii)] Right multiplication of a column by some $\pm g$ with $g\in G$.
\item[(iv)] Left multiplication of a row by some $\pm g$ with $g\in G$.
\en
\item If $\G=\ll x\rr$ is an infinite cyclic group and if $A$ is a square matrix  over $\R[\G]=\R[x^{\pm 1}]$ such that the ordinary determinant 
$\det(A)\in \R[x^{\pm 1}]$ is non-zero, then it follows from \cite[Example~3.22]{Lu02} and \cite[Section~1.2]{Rai12} that
\be \label{equ:dettwomahler} \det_{\ll x\rr}(A)=m(\det(A))\ee
where $\det(A)\in \R[x^{\pm 1}]$ is the usual determinant of $A$ over the ring $\R[x^{\pm 1}]$, and where given a non-zero polynomial $p(x)\in \R[x^{\pm 1}]$ we denote by $m(p(x))$ its Mahler measure. Recall, that if $p(x)=c_nx^n+c_{n-1}x^{n-1}+\dots+c_1x+c_0$ (with $c_n\ne 0$ and $c_0\ne 0$) and if  $a_1,\dots,a_n$ are the roots of $p(x)$, then it follows from Jensen's formula that
\be\label{equ:jensen} m(p(x)) = |c_n| \cdot \prod\limits_{j=1}^n \mbox{max}(|a_j|,1).\ee
We refer to \cite{SiW04} for more details and references.
\item For any two square matrices $A$ and $B$ of determinant class of the same size which have `full rank'  we have $\det_\G(AB)=\det_\G(A)\cdot \det_\G(B)$. 
Here, we say that a square $k\times k$-matrix over $\R[\G]$ has `full rank' if the $L^2$-Betti number of the kernel of the corresponding automorphism of $\R[\G]^k$ is zero.
\en

After this quick low-carb introduction to the Fuglede-Kadison determinant we now turn to $L^2$-torsions of chain complexes over group rings.
In the following  let
\[ C_*=\left( \,\,0\to \R[\G]^{n_k}\xrightarrow{A_{k}} \R[\G]^{n_k}\xrightarrow{A_{k-1}}
\dots  \R[\G]^{n_1}\xrightarrow{A_{1}} \R[\G]^{n_0}\to 0\right)\]
be a chain complex of free  based left-$\R[\G]$-modules. (In particular we view the elements of $\R[G]^{n_i}$ as row vectors on which the matrices $A_i$ acts by right multiplication.) We can then consider the corresponding $L^2$-Betti numbers $b_i^{(2)}(C_*)\in \R_{\geq 0}$.
 If any  of these $L^2$-Betti number is non-zero or if any of the boundary maps is not of determinant class, then we define the 
 $L^2$-torsion $\tautwo(C_*):=0$. Otherwise we define the $L^2$-torsion of $C_*$ to be 
\[ \tautwo(C_*):=\prod_{i=1}^k \det_\G(A_i)^{(-1)^i}\in \R_{>0}.\]

\subsection{The  $L^2$-torsion of a knot}

Let $K\subset S^3$ be an oriented knot and let $\g\colon \pi_K\to \G$  be an
epimorphism onto a group.
 We  consider the chain complex
\[ C_*^{\g}(X_K;\R[\G]):=\R[\G]\otimes_{\Z[\pi_K]} C_*(\wti{X_K}) \]
where $g\in \pi_K$ acts on $C_*(\wti{X_K})$ using the deck transformation action and where $g\in \pi_K$ acts on $\R[\G]$ by right-multiplication by
$\g(g)$.

\begin{remark}
The astute reader will have noticed that in contrast to the previous twisted invariants we now view $C_*(\wti{X_K})$ as a left $\Z[\pi_K]$-module. In fact it makes basically no difference whether we view $C_*(\wti{X_K})$ as a left $\Z[\pi_K]$-module or a right $\Z[\pi_K]$-module. But in most papers on twisted and higher-order Alexander torsions the later convention is used, whereas in all papers on $L^2$-Alexander torsions the former convention is used.
\end{remark}

We once again pick a  lift of each cell in $X_K$ to $\wti{X_K}$. Note that the chain complex $ C_*^{\g}(X_K;\R[\G])$ is  a chain complex of free left-$\R[\G]$-complexes where the lifts of the cells give naturally rise to a basis. We can therefore define
\[ \tautwo(K,\g):=\tautwo\left(\R[\G]\otimes_{\Z[\pi_K]}C_*(\wti{X_K}) \right)\in \R_{\geq 0}.\]
Note that $\tautwo(K,\g)\in \R_{\geq 0}$ is well-defined with no indeterminacy. 
We refer to  $\tautwo(K,\g)$ as the \emph{$L^2$-torsion of $(K,\g)$}. 
If $\g$ is the identity, then we write $\tautwo(K)=\tautwo(K,\g)$, and we refer to is as the \emph{full $L^2$-torsion of $K$}. 

As we mentioned in the introduction, L\"uck--Schick \cite{LS99} showed that  the full $L^2$-torsion of $K$ is in fact a repackaging of the volume of a knot, namely they showed that 
\be \label{equ:tauvol} \tautwo(K)=\exp\left(\frac{1}{6\pi}\vol(K)\right).\ee

\subsection{The  $L^2$-Alexander torsion of a knot}

In the last section we stated  that the volume of a knot can be recovered as an $L^2$-invariant of its exterior. Even though this is a very pretty fact, this result is also a little disappointing in so far as it shows that the full $L^2$-torsion does not give us any new information on a knot. 
The problem with $L^2$-torsions is furthermore that they are `just' numbers, they therefore have little structure.

Following Li--Zhang \cite{LiZ06a,LiZ06b,LiZ08} we will address these issues by twisting the representations using  characters.
More precisely, let $K\subset S^3$ be an oriented knot and 
let $\g\colon \pi_K\to \G$  be an 
admissible homomorphism, i.e., $\g$ is an homomorphism
such that the abelianization $\phi_K$ factors through $\g$.
 Furthermore let $t\in \R_{>0}$. 
We then consider the  {representation}
\[ \ba{rcl} \g_t\colon \pi_K&\to & \aut(\R[\G]) \\
g&\mapsto & (f\mapsto t^{\phi(g)}\g(g)\cdot f)\ea \]
and we  consider the chain complex
\[ C_*^{\g_t}(X_K;\R[\G]):= \R[\G]\otimes_{\Z[\pi_K]}C_*(\wti{X_K})\]
where $g\in \pi_K$ acts on $\R[\G]$ on the right via the representation $\g_t$.
We pick a  lift of each cell in $X_K$ to $\wti{X_K}$. 
For each $t\in \R_{>0}$ the above is now a chain complex of free left $\R[\G]$-modules where the lifts of the cells give naturally rise to a basis. We can therefore define
\[ \tautwo(K,\g)(t):=\tautwo\left(C_*^{\g_t}(X_K;\R[\G])\right)\in \R_{\geq 0}.\]
Put differently, we just  defined a function 
\[ \ba{rcl} \tautwo(K,\g)\colon \R_{>0}&\to& \R_{\geq 0}\\
t&\mapsto &\tautwo(K,\g)(t)\ea\]
which we will refer to as the \emph{$L^2$-Alexander torsion of $(K,\g)$}.
This function is well-defined up to multiplication by a function of the form $t\mapsto t^k$, for some $k\in \Z$, see e.g.~\cite{LiZ06a, LiZ06b, DubW13}.

We now say that two functions $f,g:\R_{>0}\to [0,\infty)$ are \emph{equivalent},
written as $f\doteq g$,  if there exists a $k\in \Z$, such that
\[ f(t)= t^k g(t) \mbox{ for all }t\in \R_{>0}.\]
By the above, the equivalence class of $\tautwo(K,\g)(t)$ is a well--defined invariant of $(K,\g)$.  

One of the reasons why $L^2$-invariants are so interesting is the fact that any group admits a canonical epimorphism onto a group, namely the identity map.
In the following we refer to the corresponding  $L^2$-Alexander torsion as the 
\emph{full}  $L^2$-Alexander torsion of $K$.
More precisely, the full $L^2$-Alexander torsion of $K$ is defined as 
\[ \tautwo(K)(t):=\tautwo(K,\id_{\pi_K})(t).\]
Note that in contrast to the twisted Reidemeister torsion and the higher-order Alexander torsion which always depend on the choice of a representation, the full $L^2$-Alexander torsion
 is finally again a canonical invariant. Moreover, the invariant  looks promising as it `sees the whole group'. 

The full $L^2$-Alexander torsion of $K$ is a slight variation
on the $L^2$-Alexander invariant $\Delta^{(2)}_K(t)\colon \R_{>0}\to \R_{\geq 0}$ which was first introduced
by Li--Zhang \cite{LiZ06a,LiZ06b,LiZ08}. In fact in \cite{DubFL14a} we show that 
\[ \Delta^{(2)}_K(t)\doteq \tautwo(K)(t)\cdot \max\{1,t\}.\]

\subsection{Properties}

We have seen that the Alexander polynomial and its generalizations are symmetric.
In \cite{DubFL14b} we 
will show that if $K\subset S^3$ is an oriented knot and $\g\colon \pi_K\to \G$ an admissible epimorphism onto a group, then $\tautwo(K,\g)(t)$ is also symmetric,  in the sense that for any 
representative $\tautwo(K,\g)(t)$ of the $L^2$-Alexander torsion of $(K,\g)$ we have
\be \label{equ:l2sym} \tautwo(K,\g)(t^{-1}) = t^{n}\cdot \tautwo(K,\g)(t)\ee
for some odd $n$.

It is obvious from the definitions that  the full $L^2$-Alexander torsion 
of a knot $K$ satisfies $\tautwo(K)(1)=\tautwo(K)$. Combining this with 
Equality (\ref{equ:tauvol}) we obtain that
\be \label{equ:tauvol2} \tautwo(K)(1)=\exp\left(\frac{1}{6\pi}\vol(K)\right).\ee
It follows in particular that $\tautwo(K)(1)$ is non-zero. 
The full $L^2$-Alexander torsion thus can be viewed as a deformation of the volume of a knot.

The  $L^2$-Alexander torsion has been determined only for a small number of knots.
First of all, a straightforward calculation, using 
Equations (\ref{equ:dettwomahler}) and (\ref{equ:jensen})
 shows that 
\[ \tautwo(\mbox{unknot})(t)\doteq \max\{1,t\}^{-1}.\] 
Dubois--Wegner \cite{DubW10,DubW13} used Fox differential calculus to  generalize this result and to show implicitly that if
 $T_{p,q}$ denotes the $(p,q)$-torus knot, where $p, q$ are positive coprime integers, then for any admissible epimorphism 
 $\g\co \pi_K\to \G$ we have
\be \label{equ:dw} \tautwo(T_{p,q},\g)(t)\doteq \max\{1,t^{(p-1)(q-1)-1}\}= \max\{1,t^{2\genus(K)-1}\}.\ee
In \cite{BA13a, BA13b}  Ben Aribi  also used Fox differential calculus to prove
 a  formula for the  $L^2$-Alexander torsion
for cables of  knots. Furthermore, Ben Aribi \cite[Theorem~1.2]{BA13b}  shows that the full $L^2$-Alexander torsion behaves well under the connected sum operation. More precisely, he showed that for any two oriented knots $J$ and $K$  the following equality holds
\be \label{equ:sum} \tautwo(J\# K)(t)\doteq \tautwo(J)(t)\cdot \tautwo(K)(t)\ee
which can be viewed as an analogue of Equation~(\ref{equ:deltasum}).
These results will be generalized in \cite{DubFL14a} to the study of $L^2$-Alexander torsions of graph manifolds and 3-manifolds with non-trivial JSJ decomposition.

The two results of Ben Aribi, together with Equation (\ref{equ:tauvol}) and 
\cite[Corollary~4.2]{Go83} imply  that the full $L^2$-torsion of a non-trivial knot is not equivalent to 
$\max\{1,t\}^{-1}$. Put differently, we see that the work of L\"uck--Schick and Ben Aribi implies that the full $L^2$-Alexander torsion detects the unknot:

\begin{theorem}
A knot $K\subset S^3$ is the trivial knot if and only if $\tautwo(K)(t)\doteq \max\{1,t\}^{-1}$. 
\end{theorem}

\subsection{The $L^2$-Alexander torsion and fibered knots}
Let $K$ be a fibered knot.
We denote by 
$f\co \S\to \S$ the corresponding monodromy of the fiber surface $\S$.
We can associate to $f$ its entropy $h(f)\in \R_{\geq 0}$, as defined say in  
\cite[p.~185]{FLP79}. 
Since a fibered knot admits a unique fibration we can define $h(K):=h(f)$.

Note that if $K$ is an iterated torus knot, then $K$ is fibered with $h(K)=0$.
On the other hand, if $K$ is a hyperbolic fibered knot, then the monodromy
 is pseudo-Anosov, and by \cite[p.~195]{FLP79} the entropy equals the logarithm of the dilatation of $f$.

The following theorem is proved in \cite{DubFL14a}. It gives a partial computation of the $L^2$-Alexander torsion for fibered knots.

\begin{theorem} \label{thm:tautwofibered}
Suppose $K$ is a fibered knot and $\g\co \pi_K\to \G$ is an admissible epimorphism onto a group that is residually amenable (e.g. $\Gamma$ is solvable or a 3-manifold group). Then there exists a  representative $\tautwo(K,\g)(t)$ of the $L^2$-Alexander torsion of $(K,\g)$ such that for $T:=\exp(h(K))$ we have
\[ \tautwo (K,\g)(t)=\left\{\ba{ll} 1, &\mbox{ if }t<\frac{1}{T}\\ t^{2\genus(K)-1}, &\mbox{ if }t>T.\ea \right.\]
\end{theorem}

In the following we say that a function $f\co \R_{>0}\to [0,\infty)$ is \emph{monomial in the limit}
if there exist $d,D\in \R$ and non-zero real numbers $c,C$  such that
\[ \lim_{t\to 0} \frac{f(t)}{t^{d}}=c \mbox{ and } \lim_{t\to \infty} \frac{f(t)}{t^D}=C.\]
For a function that is monomial in the limit as above we define its \emph{degree} as
\[ \deg (f(t))=D-d.\]
We furthermore refer to $C$ as the \emph{top coefficient of $f$} and 
we refer to $c$ as the \emph{bottom coefficient of $f$}.
Finally we say that $f$ is \emph{monic} if its top and bottom coefficient are both equal to $1$.
Note  that the above definitions do not depend on the equivalence class of the function. 

With these definitions we can formulate the following corollary to Theorem  \ref{thm:tautwofibered}.

\begin{corollary}
If $K$ is a fibered knot and if $\g\co \pi_K\to \G$ is an admissible epimorphism, then  $\tautwo(K,\g)(t)$ is monomial in the limit, it is monic and it has degree $2\genus(K)-1$.
\end{corollary}

\subsection{The $L^2$-Alexander torsion and the knot genus}
Let  $f\co \R^+\to [0,\infty)$ be a function which is non-zero for 
sufficiently small $t$ and for sufficiently large $t$. We then define
\[ \deg(f)= \limsup_{t \to \infty} \frac{\ln(f(t))}{\ln(t)}-\liminf_{t \to 0} \frac{\ln(f(t))}{\ln(t)}\in \R. \]
Note that for functions that are monomial in the limit this definition of degree is the same as in the previous section.

In \cite{DubFL14a} we  prove the following theorem that can be viewed as a generalization of Inequality (\ref{equ:degtau}).

\begin{theorem}\label{thm:lowerboundgenus} 
For any oriented knot $K$ and any admissible epimorphism $\pi_K\to \G$ onto a virtually abelian group the following inequality holds
\[ \deg\left(\tautwo(K,\g)(t)\right)\leq 2\genus(K)-1.\] 
\end{theorem}

\subsection{The $L^2$-Alexander torsion corresponding to the abelianization}\label{section:abelian}

From the results of the previous sections it is not easy to guess what the general structure of the $L^2$-Alexander torsions might look like. 
We therefore change the point of view and we now consider the $L^2$-Alexander torsion $\tautwo(K,\phi_K)(t)$, i.e.,  the $L^2$-Alexander torsion  corresponding 
to the abelianization map $\phi_K\colon \pi_K\to \Z$. 
In \cite{DubFL14a} we prove the following proposition which
says in particular that  the Alexander polynomial $\Delta_K(t)$ determines $\tautwo(K,\phi_K)(t)$.

\begin{proposition}\label{prop:l2ordinaryalex}
Let $K$ be a knot and let $\Delta(z)=\Delta_K(z)\in \Z[z^{\pm 1}]$ be a representative of the Alexander polynomial of $K$.
We write
\[ \Delta(z)=C\cdot z^m\cdot \prod_{i=1}^k (z-a_i),\]
with some $C\in \Z\sm \{0\},m\in \Z$ and $a_1,\dots,a_k\in\C\sm \{0\}$.
Then
\[ \tautwo(K,\phi_K)(t)\doteq C\cdot \prod_{i=1}^k \max\{|a_i|,t\}\cdot \max\{1,t\}^{-1}.\]
\end{proposition}

\begin{proof}[Outline of the proof]
It follows relatively easily from the definitions that 
there exists a matrix $A(z)$ over $\Z[z^{\pm 1}]$ with $\det(A(z))=\Delta(z)$ and such that 
\[ \tautwo(K,\phi_K)(t)\doteq\det_{\ll z\rr}(A(tz))\cdot \det_{\ll z\rr}(tz-1)^{-1}.\]
The proposition can then be deduced  from Equalities (\ref{equ:dettwomahler})
and (\ref{equ:jensen}).
\end{proof}

With these definitions we can formulate the following immediate corollary 
to Proposition \ref{prop:l2ordinaryalex}.

\begin{corollary} \label{cor:l2ordinaryalex}
Let $K$ be a knot, then $\tautwo(K,\phi_K)(t)$ is monomial in the limit with
\[ \deg\left(\tautwo(K,\phi_K)(t)\right)=\deg\left(\Delta_K(t)\right)-1.\]
Furthermore,  $\tautwo(K,\phi_K)(t)$ is monic if and only if $\Delta_K(t)$ is monic.
\end{corollary}

We return to the torus knots as an instructive example.
Recall that if  $K=T_{p,q}$ is the $(p,q)$-torus knot, where $p, q$ are coprime positive integers, then its Alexander polynomial is given by
\[ \Delta_{T_{p,q}}(t)=\frac{(t^{pq}-1)(t-1)}{(t^p-1)(t^q-1)}.\]
This is a polynomial of degree $(p-1)(q-1)$ and all the zeros are roots of unity.
Proposition \ref{prop:l2ordinaryalex} thus states that 
\[ \tautwo(K,\phi_K)(t)\doteq \max\{1,t^{(p-1)(q-1)-1}\}.\]
In fact we had already seen in Equality (\ref{equ:dw})  that this equality holds for $\tautwo(K,\g)$ and any 
 admissible epimorphism $\g$.
Note that if  we consider the torus knots $T_{3,7}$ and $T_{4,5}$, then it is now straightforward to see that all the $L^2$-Alexander torsions agree, but that  the ordinary Alexander polynomials are different.

\subsection{Questions and conjectures}

The calculations presented in (\ref{equ:dw}), Theorem \ref{thm:tautwofibered}
and Proposition \ref{prop:l2ordinaryalex} are pretty much the only calculations of  $L^2$-Alexander torsions of knots known to the authors.
We will not be intimidated by the scarcity of examples not to state a long list of ambitious conjectures.

We start out with the discussion what the functions $\tautwo(K,\g)(t)$ can look like.
We have the following conjecture.

\begin{conjecture}\label{conj:monomial} 
Let $K$ be an oriented knot and let $\g\co \pi_K\to \G$ be an admissible epimorphism. Then the following hold:
\bn[label=(\alph*)]
\item[$(a)$] $\tautwo(K,\g)(t)$ is monomial in the limit.
\item[$(b)$] $\tautwo(K,\g)(t)$ is continuous.
\item[$(c)$] $\tautwo(K,\g)(t)\cdot \max\{1,t\}$ is convex.
\en
\end{conjecture}

In order to simplify the discussion we will for the remainder of this section 
assume that Conjecture \ref{conj:monomial} (a) holds, so that we can throughout this section talk of degree, monicness, top coefficient and bottom coefficient.
(Note that by Equation (\ref{equ:l2sym}) the top and bottom coefficient of $L^2$-Alexander torsions are always the same.)

The following conjecture, which we will discuss in more detail in a future paper, can be viewed as saying that $L^2$-Alexander torsions are a generalization of the higher-order Alexander torsions. 
\medskip

\begin{conjecture}\label{conj:l2higherorder}
For any oriented knot $K$ and any epimorphism $\g\co \pi_K\to \G$ onto a non-trivial torsion-free elementary-amenable group we have
\[ \deg\left(\tautwo(K,\g)(t)\right)=\deg(\tau(K,\g)).\] 
\end{conjecture}

In Theorem \ref{thm:higherorderincrease} we saw that `the bigger the quotient - the better the lower bound on the genus'. One problem with higher-order Alexander torsions is that there is no maximal torsion-free elementary-amenable quotient
of a knot group. On the other hand such a problem does not exist for the $L^2$-Alexander torsion, we can just  use the identity of the fundamental group.
The corresponding full $L^2$-Alexander torsion should thus give the best possible lower bound on the genus, and we expect that it in fact always determines the knot genus. More precisely, we propose the following conjecture.

\begin{conjecture}\label{conj:detectsgenus} 
For any oriented knot $K$ the full $L^2$-Alexander torsion  satisfies
\[ \deg\left(\tautwo(K)(t)\right)=2\genus(K)-1.\] 
\end{conjecture}

We are quite optimistic about this conjecture, in fact in \cite{DubFL14a}
we use the work of Agol \cite{Ag08}, Przytycki--Wise \cite{PW12} and Wise \cite{Wi09,Wi12a,Wi12b} to show that given any oriented knot $K$ there exists an epimorphism
$\g\co \pi_K\to \G$ onto a virtually abelian group with $\deg\left(\tautwo(K,\g)(t)\right)=2\genus(K)-1$.
\medskip

Theorem \ref{thm:tautwofibered} also shows that $L^2$-Alexander torsions have information about fiberedness. We therefore ask the following question.

\begin{question}\label{question:detectsfib} 
If $K$ is an oriented knot  such that  the full $L^2$-Alexander torsion $\tautwo(K)(t)$ is monic, does this imply that $K$ is fibered?
\end{question}

The attentive reader might have noticed that this is phrased as a question, thus somewhat less optimistically than our conjectures. In fact there is some computational evidence, due to Ben Aribi, that the answer might in fact be no.

If $K$ is not fibered, then it is an interesting problem to determine what (if any) geometric or dynamic  information is contained in the top coefficient of $\tautwo(K)(t)$. 
We now turn to the case that $K$ is fibered. 
We conjecture that the full 
 $L^2$-Alexander torsion determines the dilatation of the monodromy.
 More precisely, we propose the following conjecture.

\begin{conjecture}\label{conj:dilatation}
Let $K$ be a fibered oriented knot with monodromy $f$. We normalize $\tautwo(K)(t)$ such that $\tautwo(K)(1)=1$. Then
\[ \sup\left\{T\in \Q_{\geq 0}\,|\, \tautwo(K)(t)|_{(0,T)}\mbox{ is constant}\right\}=\exp(-h(f)).\]
\end{conjecture}

Note that by Theorem \ref{thm:tautwofibered} we know that the inequality `$\geq$' holds. Proving the inequality `$\leq$' seems significantly harder.

We have two  pieces of evidence for the  inequality `$\leq$' of the conjecture. First of all,
it follows from Equation (\ref{equ:dw}) that it holds if the monodromy has finite order, i.e., if $K$ is a torus knot.
 Furthermore, if we consider 
$\tautwo(K,\phi_K)(t)$ with the  normalization  $\tautwo(K)(1)=1$,  then
it follows fairly easily from Corollary \ref{cor:l2ordinaryalex} that 
\[ \sup\left\{T\in \Q_{\geq 0}\,|\, \tautwo(K,\phi_K)|_{(0,T)}\mbox{ is constant}\right\}=\exp(-m),\]
where $m$ is the maximal absolute value of an eigenvalue of the induced 
automorphism $f_*$ of $H_1(\S;\R)$. But it is well-known e that $m=h(f_*\co H_1(\S;\Z)\to H_1(\S;\Z))$. 
\medskip

In Section \ref{section:abelian} we saw that $L^2$-Alexander torsions can not distinguish certain pairs of torus knots. 
Our final question now asks whether at least hyperbolic knots are determined by their full $L^2$-Alexander torsion. Here one has to be a little careful about what one means by `determine'. We say that two knots $J$ and $K$ are equivalent if there exists a diffeomorphism $h$ of $S^3$ with $h(J)=K$ as sets, i.e. we do not demand that the orientations match. It is straightforward to see that the full $L^2$-Alexander torsions of equivalent knots are equivalent functions.
We propose the following question.

\begin{question}
If $J$ and $K$ are  two oriented hyperbolic knots with equivalent  full $L^2$-Alexander torsions,
does this imply that $J$ and $K$ are equivalent?
\end{question}

A typical example for pairs of knots which are difficult to distinguish is given by mutants.
Many invariants, e.g. the Alexander polynomial and the Jones polynomial do not distinguish mutants. It follows from  work of Ruberman \cite{Ru87}
and Equation (\ref{equ:tauvol2}) that the evaluation of the full $L^2$-Alexander torsion at $t=1$ stays invariant under mutation. On the other hand the genus is not invariant under mutation,
for example the Conway knot has genus 3 and its mutant, the Kinoshita-Terasaka knot, has genus 2. 
In light of Conjecture \ref{conj:detectsgenus}  we thus expect that in general the full $L^2$-Alexander torsion is not invariant under mutation.

\end{document}